\documentclass[leqno,11pt,a4paper]{amsart}
\usepackage[usenames,dvipsnames]{color}
\usepackage{hyperref}
\usepackage{ragged2e}
\usepackage{multirow}
\newtheorem{thm}{Theorem}[section]
\newtheorem{lem}[thm]{Lemma}
\newtheorem{cor}[thm]{Corollary}
\newtheorem{prop}[thm]{Proposition}
\theoremstyle{definition}
\newtheorem{defn}[thm]{Definition}
\theoremstyle{remark}

\theoremstyle{remark}
\newtheorem{eg}[thm]{Example}
\numberwithin{equation}{section}
\newenvironment{ack}{\bigskip\noindent\textbf{Acknowledgments.}}{}
\long\def\blankfootnotetext#1{\begingroup\def\thefootnote{\fnsymbol{footnote}}\footnotetext{#1}\endgroup}
\newcommand{\Z}{\mathbb{Z}}

\newcommand{\R}{\mathbb{R}}
\newcommand{\C}{\mathbb{C}}
\newcommand{\abs}[1]{\left\vert{#1}\right\vert}
\newcommand{\vol}[1]{\mathrm{vol}\!\left(#1\right)}
\renewcommand{\dim}[1]{\mathrm{dim}\!\left(#1\right)}
\renewcommand{\Re}[1]{\mathrm{Re}\!\left(#1\right)}

\begin{document}
%-------------------------------------------------------------------------------
\author[G.~Heged{\"u}s]{G{\'a}bor Heged{\"u}s}
\address{Johann Radon Institute for Computational and Applied Mathematics\\Austrian Academy of Sciences\\Altenbergerstra{\ss}e 69\\A-$4040$ Linz\\Austria}
\email{gabor.hegedues@oeaw.ac.at}
\author[A.~M.~Kasprzyk]{Alexander M.~Kasprzyk}
\address{School of Mathematics and Statistics\\University of Sydney\\Sydney\ NSW $2006$\\Australia}
\email{a.m.kasprzyk@usyd.edu.au}
%-------------------------------------------------------------------------------
%-------------------------------------------------------------------------------
\title{Roots of Ehrhart Polynomials of Smooth Fano Polytopes}
%-------------------------------------------------------------------------------
\begin{abstract}
V.~Golyshev conjectured that for any smooth polytope $P$ with $\dim{P}\le 5$ the roots $z\in\C$ of the Ehrhart polynomial for $P$ have real part equal to~$-1/2$. An elementary proof is given, and in each dimension the roots are described explicitly. We also present examples which demonstrate that this result cannot be extended to dimension six.
\end{abstract}
%-------------------------------------------------------------------------------
\maketitle
\blankfootnotetext{Research supported in part by OTKA grant K77476.}
%-------------------------------------------------------------------------------
\section{Introduction}
%-------------------------------------------------------------------------------
Let $P$ be a $d$-dimensional convex lattice polytope in $\R^d$. Let $L_P(m):=\abs{mP\cap\Z^d}$ denote the number of lattice points in $P$ dilated by a factor of $m\in\Z_{\geq 0}$. In general the function $L_P$ is a polynomial of degree $d$, called the Ehrhart polynomial~\cite{Ehr67}.

The roots of Ehrhart polynomials have recently been the subject of much study (for example~\cite{BHW07,BD08,HHO10,Pfe07}), with a significant portion of this work being based on exhaustive computer calculations using the known classifications of polytopes. It has been conjectured in~\cite{BLDPS05} that if $z\in\C$ is a root of $L_P$, then the real part $\Re{z}$ is bounded by $-d\leq\Re{z}\leq d-1$; Braun has shown~\cite{Bra08} that $z$ lies inside the disc centred at $-1/2$ of radius $d(d-1/2)$.

\begin{defn}\label{defn:reflexive}
A convex lattice polytope $P$ is called \emph{reflexive} if the dual polytope
$$P^\vee:=\{u\in\R^d\mid \left<u,v\right>\le 1\text{ for all }v\in P\}$$
is also a lattice polytope.
\end{defn}

There are many interesting and well-known characterisations of reflexive polytopes (for example~\cite[Theorem~3.5]{HK10}). They are of particular relevance to toric geometry: reflexive polytopes correspond to Gorenstein toric Fano varieties (see~\cite{Bat94}) and have been classified up to dimension four.

Any reflexive polytope $P$ satisfies
\begin{equation}\label{eq:reflexive_layers}
L_P(m)=L_{\partial P}(m) + L_P(m-1) \text{ for all } m\in\Z_{>0},
\end{equation}
where $\partial P$ denotes the boundary of $P$. As a consequence, Macdonald's Reciprocity Theorem~\cite{Mac71} tells us that $L_P(-m-1)=(-1)^dL_P(m)$. In particular we observe that the roots of $L_P$ are symmetrically distributed with respect to the line $\Re{z}=-1/2$.

\begin{thm}[\protect{\cite[Proposition~1.8]{BHW07}}]
Let $P$ be a $d$-dimensional convex lattice polytope such that for all roots $z$ of $L_P$, $\Re{z}=-1/2$. Then, up to unimodular translation, $P$ is a reflexive polytope with $\vol{P}\le 2^d$.
\end{thm}

\begin{thm}[\protect{\cite[Theorem~0.1]{HHO10}}]
In each dimension $d$ there exists a reflexive polytope $P$ such that if $z\in\C\setminus\R$ is a root of $L_P$ then $\Re{z}=-1/2$.
\end{thm}

\begin{defn}\label{defn:smooth}
A $d$-dimensional convex lattice polytope $P$ is called \emph{smooth} if the vertices of any facet of $P$ form a $\Z$-basis of the ambient lattice $\Z^d$.
\end{defn}

Clear any smooth polytope is simplicial and reflexive. Smooth polytopes are in bijective correspondence with non-singular toric Fano varieties, and have been classified up to dimension eight~\cite{Obr07}.

V.~Golyshev conjectured in~\cite[\S5]{Gol09} that, for any smooth polytope $P$ of dimension $d\le 5$, the roots $z\in\C$ of $L_P$ satisfy $\Re{z}=-1/2$ (the ``canonical line hypothesis''). Notice that it is not required that $z\notin\R$. We prove Golyshev's conjecture without resorting to the known classifications -- see Sections~\ref{sec:dim_2_and_3} and~\ref{sec:dim_4_and_5} below.

\begin{thm}[Golyshev]\label{thm:Golyshev}
Let $P$ be a smooth polytope of dimension $d\leq 5$. If $z\in\C$ is a root of $L_P(m)$ then $\Re{z}=-1/2$. 
\end{thm}

Explicit descriptions of the roots are given in Corollaries~\ref{cor:explicit_dim_2_and_3} and~\ref{cor:explicit_dim_4_and_5}. We summarise them in the following theorem.

\begin{thm}\label{thm:explicit}
Let $P$ be a smooth $d$-dimensional polytope, and suppose that $z=-1/2+\beta i\in\C$ is a root of $L_P$. If $d=2$ then
$$\beta^2=-\frac{1}{4}+\frac{2}{f_0}.$$
If $d=3$ then $\beta=0$ or
$$\beta^2=-\frac{1}{4}+\frac{6}{f_0-2}.$$
If $d=4$ then
$$\beta^2=-\frac{17}{4}+\frac{3b_2}{b_2-2f_0}\pm\sqrt{1-\frac{12(f_0+2)}{b_2-2f_0}+\frac{36f_0^2}{(b_2-2f_0)^2}}.$$
If $d=5$ then $\beta=0$ or
$$\beta^2=-\frac{5}{4}+\frac{10(f_0-2)}{6+b_2-4f_0}\pm\sqrt{1-\frac{20(f_0+4)}{6+b_2-4f_0}+\frac{100(f_0-2)^2}{(6+b_2-4f_0)^2}}.$$
\end{thm}

The following example demonstrates that we cannot extend Theorem~\ref{thm:Golyshev} to dimension~$6$.

\begin{eg}
There exist exactly four smooth polytopes in dimension six having roots $z$ of the Ehrhart polynomial such that $\Re{z}\ne-1/2$; in each case $z\notin\R$. The polytopes have IDs $1895$, $1930$, $4853$, and $5817$ in the Graded Ring Database\footnote{\href{http://grdb.lboro.ac.uk/search/toricsmooth?id_cmp=in&id=1895,1930,4853,5817}{\texttt{http://grdb.lboro.ac.uk/search/toricsmooth?id\_cmp=in\&id=1895,1930,4853,5817}}}. The corresponding Ehrhart polynomials are:
\begin{align*}
&1+\frac{31}{10}m+\frac{257}{60}m^2+\frac{5}{2}m^3+\frac{19}{12}m^4+\frac{2}{5}m^5+\frac{2}{15}m^6,\\
&1+\frac{7}{2}m+\frac{175}{36}m^2+\frac{35}{12}m^3+\frac{35}{18}m^4+\frac{7}{12}m^5+\frac{7}{36}m^6,\\
&1+\frac{7}{2}m+\frac{21}{4}m^2+\frac{15}{4}m^3+\frac{5}{2}m^4+\frac{3}{4}m^5+\frac{1}{4}m^6,\\
&1+\frac{31}{10}m+\frac{257}{60}m^2+\frac{5}{2}m^3+\frac{19}{12}m^4+\frac{2}{5}m^5+\frac{2}{15}m^6.
\end{align*}
The second polytope has roots where $\Re{z}>0$, and where $\Re{z}<-1$. This demonstrates that the more general ``canonical strip hypothesis'' does not hold in dimension six.
\end{eg}

%-------------------------------------------------------------------------------
\section{Dimensions Two and Three}\label{sec:dim_2_and_3}
%-------------------------------------------------------------------------------
One of the fundamental pieces of numerical data associated with a polytope is the $f$-vector, which enumerates the number of faces of $P$. We begin by deriving an expression for the Ehrhart polynomial of a smooth polytope in terms of its $f$-vector.

\begin{defn}\label{defn:f_vector}
Let $P$ be a $d$-dimensional convex polytope. Define $f_{-1}:=1$, $f_d:=1$, and $f_i$ equal to the number of $i$-dimensional faces of $P$, for any $0\le i\le d-1$. The \emph{$f$-vector} of $P$ is the sequence $(f_{-1},f_0,\ldots,f_d)$.
\end{defn}

\begin{lem}\label{lem:smooth_point_counting}
Let $P$ be a $d$-dimensional smooth polytope. Then
$$L_P(m)=\sum_{i=-1}^{d-1}f_i{m\choose i+1}\quad\text{and}\quad L_{\partial P}(m)=\sum_{i=0}^{d-1} f_i{m-1\choose i}.$$
\end{lem}
\begin{proof}
Clearly
$$L_{\partial P}(m)=f_0+\sum_F\abs{(mF)^\circ\cap\Z^d},$$
where the sum is taken over all $i$-dimensional faces $F$ of $P$, $i>0$, and $Q^\circ$ denotes the (relative) interior of $Q$. Since $P$ is smooth, $F\cap\Z^d$ forms part of a basis for the underlying lattice $\Z^d$ for any face $F$. Hence
$$L_{\partial P}(m)=\sum_{i=0}^{d-1}f_i{m-1\choose i}.$$

To calculate $L_P(m)$ we make use of~\eqref{eq:reflexive_layers}:
\begin{align*}
L_P(m)&=1+\sum_{k=1}^mL_{\partial P}(k)=1+\sum_{k=1}^m\sum_{i=0}^{d-1}f_i{k-1\choose i}\\
&=1+\sum_{i=0}^{d-1}f_i\sum_{k=1}^m{k-1\choose i}\\
&=1+\sum_{i=0}^{d-1}f_i{m\choose i+1}\\
&=\sum_{i=-1}^{d-1}f_i{m\choose i+1}.
\end{align*}
\end{proof}

The $f$-vectors of low-dimensional smooth polytopes were calculated in~\cite[Theorem~4.2]{HK10}. As a consequence we obtain the following formulae for the Ehrhart polynomial:

\begin{cor}\label{cor:smooth_Ehrhart_polys}
Let $P$ be a $d$-dimensional smooth polytope. Define $b_2:=\abs{\partial(2P)\cap\Z^d}$.\\
If $d=2$ then
$$L_P(m)=1+\frac{1}{2}f_0m+\frac{1}{2}f_0m^2.$$
If $d=3$ then
$$L_P(m)=1+\frac{1}{6}(f_0+10)m+\frac{1}{2}(f_0-2)m^2+\frac{1}{3}(f_0-2)m^3$$
If $d=4$ then
$$L_P(m)=1+\frac{1}{12}(8f_0-b_2)m+\frac{1}{24}(14f_0-b_2)m^2-\frac{1}{12}(2f_0-b_2)m^3-\frac{1}{24}(2f_0-b_2)m^4.$$
If $d=5$ then
\begin{align*}
L_P(m)=1+\frac{1}{60}(14f_0-b_2+94)m+\frac{1}{24}&(16f_0-b_2-30)m^2+\frac{1}{3}(f_0-2)m^3\\
&-\frac{1}{24}(4f_0-b_2-6)m^4-\frac{1}{60}(4f_0-b_2-6)m^5.
\end{align*}
\end{cor}

Casagrande provides sharp bounds on the number of vertices $f_0$ of a smooth polytope in terms of the dimension:

\begin{thm}[\cite{Cas04}]\label{thm:Casagrande}
Let $P$ be a $d$-dimensional smooth polytope. Then
$$f_0\leq\left\{\begin{array}{ll}
3d,&\text{ if $d$ is even;}\\
3d-1,&\text{ if $d$ is odd.}
\end{array}\right.$$
\end{thm}

We now prove Theorem~\ref{thm:Golyshev} without resorting to the classifications in dimensions~$2$ and~$3$.

\begin{prop}\label{prop:real_part_dim_2_3}
Let $P$ be a smooth polytope of dimension two or three. If $z\in\C$ is a root of $L_P(m)$ then $\Re{z}=-1/2$.
\end{prop}
\begin{proof}
$d=2$:
By Corollary~\ref{cor:smooth_Ehrhart_polys} we know that
$$L_P(m)=1+\frac{1}{2}f_0m+\frac{1}{2}f_0m^2.$$
Let $\alpha+\beta i\in\C$ be a root of $L_P$, where $\alpha,\beta\in\R$. Assume that $\beta\ne 0$. By considering the imaginary part we obtain
$$\beta(1+2\alpha)=0,$$
hence $\alpha=-1/2$ as required. The real part simplifies to
$$\beta^2=\frac{2}{f_0}-\frac{1}{4}.$$
Theorem~\ref{thm:Casagrande} tells us that this is always positive, thus we obtain both roots of $L_P$.

$d=3$:
In this case Corollary~\ref{cor:smooth_Ehrhart_polys} tells us that
$$L_P(m)=1+\frac{1}{6}(f_0+10)m+\frac{1}{2}(f_0-2)m^2+\frac{1}{3}(f_0-2)m^3,$$
giving real and imaginary parts:
\begin{equation}\label{eq:dim_3_real}
1+\frac{1}{6}(f_0+10)\alpha+\frac{1}{2}(f_0-2)(\alpha^2-\beta^2)+\frac{1}{3}(f_0-2)(\alpha^2-3\beta^2)\alpha=0,
\end{equation}
\begin{equation}\label{eq:dim_3_imaginary}
\frac{1}{6}(f_0+10)\beta+(f_0-2)\alpha\beta+\frac{1}{3}(f_0-2)(3\alpha^2-\beta^2)\beta=0.
\end{equation}

Assume that $\beta\neq 0$. Equation~\eqref{eq:dim_3_imaginary} gives us
\begin{equation}\label{eq:dim_3_imaginary_2}
(f_0-2)\beta^2=\frac{1}{2}f_0+5+3(f_0-2)\alpha+3(f_0-2)\alpha^2.
\end{equation}
Substituting~\eqref{eq:dim_3_imaginary_2} into~\eqref{eq:dim_3_real} gives
$$\frac{1}{12}(2\alpha+1)\left(4(f_0-2)(2\alpha+1)^2+26-f_0\right).$$
Clearly $\alpha=-1/2$ is one possible solution. The discriminant of $4(f_0-2)(2\alpha+1)^2+26-f_0$, regarded as a quadratic in $2\alpha+1$, is $16(f_0-2)(f_0-26)$. This is negative when $2\leq f_0\leq 26$, and by Theorem~\ref{thm:Casagrande} this covers all possible values of $f_0$. Hence $\alpha=-1/2$ is the only solution. The values for $\beta$ are determined by~\eqref{eq:dim_3_imaginary_2}:
$$\beta^2=\frac{26-f_0}{4f_0-8}.$$

If we allow $\beta=0$ then~\eqref{eq:dim_3_real} becomes
$$\frac{1}{24}(2\alpha+1)\left((f_0-2)(2\alpha+1)^2+26-f_0\right).$$
Once more the discriminant of the quadratic component tells us that the only solution is when $\alpha=-1/2$.
\end{proof}

The proof of Proposition~\ref{prop:real_part_dim_2_3} gives us explicit equations for the roots of $L_P$.

\begin{cor}\label{cor:explicit_dim_2_and_3}
Let $P$ be a smooth $d$-dimensional polytope, and suppose that $z=-1/2+\beta i\in\C$ is a root of $L_P$. If $d=2$ then
$$\beta^2=-\frac{1}{4}+\frac{2}{f_0}.$$
If $d=3$ then $\beta=0$ or
$$\beta^2=-\frac{1}{4}+\frac{6}{f_0-2}.$$
\end{cor}

%-------------------------------------------------------------------------------
\section{Dimensions Four and Five}\label{sec:dim_4_and_5}
%-------------------------------------------------------------------------------
In order to prove Theorem~\ref{thm:Golyshev} in dimension~$4$ we require a some additional results. Throughout we write $b_2:=\abs{\partial(2P)\cap\Z^d}$, where $d$ is the dimension of $P$.

\begin{lem}[\protect{\cite[Corollary~4.4]{HK10}}]\label{lem:bounds_dim_4_1}
Let $P$ be a four-dimensional smooth polytope. Then
$$5f_0-10\le b_2\le 5f_0.$$
\end{lem}

\begin{lem}\label{lem:bounds_dim_4_2}
Let $P$ be a four-dimensional smooth polytope. Then
$$(b_2-8f_0)^2>24(b_2-2f_0).$$
\end{lem}
\begin{proof}
From Lemma~\ref{lem:bounds_dim_4_1} we have that
\begin{align*}
(b_2-8f_0)^2&=(b_2-16f_0)b_2+64f_0^2\\
&\geq(10-5f_0)(10+11f_0)+64f_0^2\\
&=9f_0^2+60f_0+100\\
&=(3f_0+10)^2
\end{align*}
Clearly $72f_0<(3f_0+10)^2$, and since $24(b_2-2f_0)\leq 72f_0$ (by Lemma~\ref{lem:bounds_dim_4_1}) we obtain the result.
\end{proof}

We shall also make use of the following trivial observation:

\begin{lem}\label{lem:deg_4_roots}
Let $g(x):=ax^4+bx^2+c\in\R[x]$ be a polynomial such that $a>0$, $b<0$, $c>0$ and $b^2-4ac>0$. Then $g$ has four distinct real roots.
\end{lem}

\begin{prop}\label{prop:real_part_dim_4}
Let $P$ be a four-dimensional smooth polytope. If $z\in\C$ is a root of $L_P(m)$ then $\Re{z}=-1/2$.
\end{prop}
\begin{proof}
In four dimensions the Ehrhart polynomial simplifies to
$$L_P(m)=1+\frac{1}{12}(8f_0-b_2)m(m+1)-\frac{1}{24}(2f_0-b_2)m^2(m+1)^2.$$
If $z=\alpha+i\beta$ is a root of $L_P$ then, by considering the real and imaginary parts, we obtain
\begin{equation}\label{eq:dim_4_real}
24+12f_0((\alpha+1)\alpha-\beta^2)-(2f_0-b_2)\alpha(\alpha+1)(\alpha(\alpha+1)-2-6\beta^2)-(2f_0-b_2)\beta^2(\beta^2+1)=0,
\end{equation}
\begin{equation}\label{eq:dim_4_imaginary}
\left(6f_0-(2f_0-b_2)\left((\alpha+1)\alpha-\beta^2-1\right)\right)(2\alpha+1)\beta=0.
\end{equation}

Clearly $\alpha=-1/2$ is a possible solution to equation~\eqref{eq:dim_4_imaginary}, in which case $\beta$ satisfies (by~\eqref{eq:dim_4_real})
\begin{equation}\label{eq:dim_4_condition}
16(b_2-2f_0)\beta^4+8(5b_2-34f_0)\beta^2+3(128+3b_2-22f_0)=0.
\end{equation}
This quadratic in $\beta^2$ has distinct real solutions if and only if
$$(b_2-8f_0)^2-24(b_2-2f_0)>0.$$
By Lemma~\ref{lem:bounds_dim_4_2} we know that this is always true.

Now we consider the signs of the coefficients of~\eqref{eq:dim_4_condition}. The leading coefficient is equal to $1/2f_2$, and so is positive. The coefficient of $\beta^2$ is always negative by Lemma~\ref{lem:bounds_dim_4_1}, and the constant term is positive by Lemma~\ref{lem:bounds_dim_4_2}. Hence, by Lemma~\ref{lem:deg_4_roots}, there are four distinct real solutions to equation~\eqref{eq:dim_4_real}.

We have found four distinct roots when $\Re{z}=-1/2$. Since $L_P$ is of degree four, we are done.
\end{proof}

Finally we consider dimension five.

\begin{lem}[\protect{\cite[Corollary~4.4]{HK10}}]\label{lem:dim_5_bound_1}
Let $P$ be a five-dimensional smooth polytope. Then
$$42f_0-105\leq 7b_2\leq 52f_0-90.$$
\end{lem}

\begin{lem}\label{lem:dim_5_bound_2}
Let $P$ be a five-dimensional smooth polytope. Then
$$100(f_0-2)^2+(6+b_2-4f_0)^2> 20(6+b_2-4f_0)(f_0+4).$$
\end{lem}
\begin{proof}
We begin by observing that the statement is equivalent to
$$\big(10(f_0-2)-(6+b_2-4f_0)\big)^2>120(6+b_2-4f_0),$$
which in turn is equivalent to
$$\big(13(f_0-2)-(b_2-f_0)\big)\big(13(f_0-2)-(b_2-f_0)+120\big)>1200(f_0-2).$$

From Lemma~\ref{lem:dim_5_bound_1} we have that
$$13(f_0-2)-(b_2-f_0)\geq\frac{46}{7}f_0-\frac{92}{7},$$
which is always positive since $f_0\geq 6$. Hence
\begin{align*}
\big(13(f_0-2)&-(b_2-f_0)\big)\big(13(f_0-2)-(b_2-f_0)+120\big)-1200(f_0-2)\\
&\geq\big(\frac{46}{7}f_0-\frac{92}{7}\big)\big(\frac{46}{7}f_0-\frac{92}{7}+120\big)-1200(f_0-2)\\
&=\frac{4}{49}(f_0-2)(529f_0-6098).
\end{align*}
This is positive for all $f_0\geq 12$.

To prove the inequality when $f_0\le 11$ we consider
\begin{align*}
\big(13(f_0-2)&-(b_2-f_0)\big)\big(13(f_0-2)-(b_2-f_0)+120\big)-1200(f_0-2)\\
&\geq\big(13(f_0-2)-(b_2-f_0)\big)\big(\frac{46}{7}f_0-\frac{92}{7}+120\big)-1200(f_0-2)\\
&=-\frac{2}{7}(23f_0+374)b_2+\frac{4}{7}(161f_0^2+219f_0-662).
\end{align*}
We wish to show that
$$-\frac{2}{7}(23f_0+374)b_2+\frac{4}{7}(161f_0^2+219f_0-662)>0$$
whenever $6\leq f_0\leq 11$. It is enough to prove that, in the given range,
\begin{equation}\label{eq:dim_5_bound_2_11}
b_2<\frac{2(161f_0^2+219f_0-662)}{23f_0+374}.
\end{equation}

Now
$$b_2-f_0=f_1\leq{f_0\choose 2},$$
and so
$$b_2\leq\frac{f_0(f_0+1)}{2}.$$
We shall show that
$$\frac{f_0(f_0+1)}{2}<\frac{2(161f_0^2+219f_0-662)}{23f_0+374}.$$
But this is trivial; the cubic
\begin{align*}
f_0(f_0+1)&(23f_0+374)-4(161f_0^2+219f_0-662)\\
&=23f_0^3-247f_0^2-502f_0+2648
\end{align*}
is negative when $6\leq f_0\leq 11$, hence equation~\eqref{eq:dim_5_bound_2_11} holds.
\end{proof}

\begin{prop}\label{prop:real_part_dim_5}
Let $P$ be a five-dimensional smooth polytope. If $z\in\C$ is a root of $L_P(m)$ then $\Re{z}=-1/2$.
\end{prop}
\begin{proof}
Let $z=\alpha+i\beta\in\C$ be a root of $L_P$, where $P$ is a five-dimensional smooth polytope. By Corollary~\ref{cor:smooth_Ehrhart_polys} we see that $\alpha$ and $\beta$ must satisfy
\begin{equation}\label{eq:dim_5_real}
\begin{split}
(2\alpha+1)\Big((6+b_2-4f_0)\big((\alpha-1)\alpha(\alpha+1)(\alpha+&2)-10(\alpha+1)\alpha\beta^2+5(\beta^2+1)\beta^2\big)+\\
&20(f_0-2)\left((\alpha+1)\alpha-3\beta^2\right)+120\Big)=0,
\end{split}
\end{equation}
\begin{equation}\label{eq:dim_5_imaginary}
\begin{split}
(14f_0-b_2+94)&\beta+5(16f_0-b_2-30)\alpha\beta+20(f_0-2)(3\alpha^2-\beta^2)\beta-\\
10&(4f_0-b_2-6)(\alpha^2-\beta^2)\alpha\beta-(4f_0-b_2-6)(5\alpha^4-10\alpha^2\beta^2+\beta^4)\beta=0.
\end{split}
\end{equation}

Clearly $\alpha=-1/2, \beta=0$ is always a solution. Suppose that $\alpha=-1/2$ and $\beta\ne 0$. Equation~\eqref{eq:dim_5_real} holds, and from~\eqref{eq:dim_5_imaginary} we obtain
\begin{equation}\label{eq:dim_5_imaginary_2}
16(6+b_2-4f_0)\beta^4+40(22+b_2-12f_0)\beta^2+2134+9b_2-116f_0=0.
\end{equation}
This quadratic in $\beta^2$ has distinct real solutions if and only if
$$100(f_0-2)^2+(6+b_2-4f_0)^2> 20(6+b_2-4f_0)(f_0+4),$$
which holds by Lemma~\ref{lem:dim_5_bound_2}.

As in the four-dimensional case we consider the signs of the coefficients of~\eqref{eq:dim_5_imaginary_2}. The leading coefficient equals $1/2f_4$ and so is positive. The coefficient of $\beta^2$ is negative by Lemma~\ref{lem:dim_5_bound_1} and the fact that $f_0\geq 6$, and the constant term is positive (again by Lemma~\ref{lem:dim_5_bound_1}). Thus, by Lemma~\ref{lem:deg_4_roots}, equation~\eqref{eq:dim_5_imaginary_2} has four distinct real solutions.

Hence we have found all five roots of $L_P$, and in each case $\Re{z}=-1/2$ as required.
\end{proof}

From equations~\eqref{eq:dim_4_condition} and~\eqref{eq:dim_5_imaginary_2} we have

\begin{cor}\label{cor:explicit_dim_4_and_5}
Let $P$ be a smooth $d$-dimensional polytope, and suppose that $z=-1/2+\beta i\in\C$ is a root of $L_P$. If $d=4$ then
$$\beta^2=-\frac{17}{4}+\frac{3b_2}{b_2-2f_0}\pm\sqrt{1-\frac{12(f_0+2)}{b_2-2f_0}+\frac{36f_0^2}{(b_2-2f_0)^2}}.$$
If $d=5$ then $\beta=0$ or
$$\beta^2=-\frac{5}{4}+\frac{10(f_0-2)}{6+b_2-4f_0}\pm\sqrt{1-\frac{20(f_0+4)}{6+b_2-4f_0}+\frac{100(f_0-2)^2}{(6+b_2-4f_0)^2}}.$$
\end{cor}

%-------------------------------------------------------------------------------
\section{Concluding Remarks}
%-------------------------------------------------------------------------------
In four dimensions one can prove Theorem~\ref{thm:Golyshev} without knowing the explicit equation for the Ehrhart polynomial. We require the following result.

\begin{prop}[\protect{\cite[Proposition~1.9]{BHW07}}]\label{prop:BHW}
Let $P$ be a four-dimensional reflexive polytope. Every root $z\in\C$ of $L_P(m)$ has $\Re{z}=-1/2$ if and only if
\begin{itemize}
\item[(i)] $2\abs{\partial P\cap\Z^4}\leq 9\,\vol{P}+16$, and
\item[(ii)] $\left(\abs{\partial P\cap\Z^4}-4\,\vol{P}\right)^2\geq 16\,\vol{P}$.
\end{itemize}
\end{prop}

\begin{proof}[Alternative proof in dimension four.]
First we show that condition~(i) of Proposition~\ref{prop:BHW} is satisfied. Since $P$ is smooth, $f_0=\abs{\partial P\cap\Z^4}$. It follows from Lemma~\ref{lem:bounds_dim_4_1} that $15f_0\leq 3b_2+30$. Hence $9f_0\leq 3(b_2-2f_0)+30$. By Theorem~\ref{thm:Casagrande} we have that $f_0\leq 12$, giving us the (very crude) inequality
\begin{equation}\label{eq:dim_4_crude}
16f_0<3(b_2-2f_0)+128.
\end{equation}
In four dimensions we have that $f_3=b_2-2f_0$ (\cite[Theorem~4.2]{HK10}) and, since $P$ is smooth, $f_3=24\,\vol{P}$. Substituting into equation~\eqref{eq:dim_4_crude} gives condition~(i).

That Proposition~\ref{prop:BHW}~(ii) holds is immediate from Lemma~\ref{lem:bounds_dim_4_2} and the fact that $b_2-2f_0=24\,\vol{P}$.
\end{proof}

Theorem~\ref{thm:explicit} tells us that in order to compute the roots of the Ehrhart polynomial we need only know $f_0$ and, in dimensions four and five, $b_2:=\abs{\partial(2P)\cap\Z^d}$. Clearly $f_0\ge d+1$, and Theorem~\ref{thm:Casagrande} provides a sharp upper bound. The values of $b_2$ can be calculated from {\O}bro's classification~\cite{Obr07}. The possible pairs $(f_0,b_2)$ are reproduced in Tables~\ref{tab:pairs_dim_4} and~\ref{tab:pairs_dim_5}.

\begin{table}[t]
\centering
\begin{tabular}{|c|cccccccccccccccccccc|}
\hline
$f_0$&5&6&6&7&7&7&8&8&8&8&9&9&9&9&9&10&10&10&11&12\\
\hline
$b_2$&15&20&21&25&26&27&31&32&33&34&36&38&39&41&42&44&45&50&52&60\\
\hline
\end{tabular}
\caption{The possible pairs $(f_0,b_2)$ for the $124$ four-dimensional smooth polytopes.}
\label{tab:pairs_dim_4}
\end{table}

\begin{table}[t]
\centering
\begin{tabular}{l}
\begin{tabular}{|c|ccccccccccccccc|}
\hline
$f_0$&6&7&7&8&8&8&8&9&9&9&9&9&10&10&10\\
\hline
$b_2$&21&27&28&33&34&35&36&40&41&42&43&44&46&49&50\\
\hline
\end{tabular}\vspace{0.3em}\\
\begin{tabular}{|c|cccccccccccccc|}
\hline
$f_0$&10&10&10&11&11&11&11&11&11&12&12&12&13&14\\
\hline
$b_2$&51&52&53&56&58&59&60&61&62&66&67&72&76&86\\
\hline
\end{tabular}
\end{tabular}
\caption{The possible pairs $(f_0,b_2)$ for the $866$ five-dimensional smooth polytopes.}
\label{tab:pairs_dim_5}
\end{table}
%-------------------------------------------------------------------------------
\begin{ack}
The authors wish to express their gratitude to Alessio Corti for alerting them to~\cite{Gol09}.
\end{ack}
%-------------------------------------------------------------------------------
\bibliographystyle{amsalpha}
\newcommand{\etalchar}[1]{$^{#1}$}
\providecommand{\bysame}{\leavevmode\hbox to3em{\hrulefill}\thinspace}
\providecommand{\MR}{\relax\ifhmode\unskip\space\fi MR }
\providecommand{\MRhref}[2]{%
  \href{http://www.ams.org/mathscinet-getitem?mr=#1}{#2}
}
\providecommand{\href}[2]{#2}

%-------------------------------------------------------------------------------
\end{document}